\documentclass[12pt, reqno]{amsart}
\usepackage{amsmath, amsthm, amscd, amsfonts, amssymb, graphicx, color}
\usepackage{setspace}
\usepackage{mathrsfs}
\usepackage{multicol}
\usepackage[bookmarksnumbered, colorlinks, plainpages]{hyperref}
\hypersetup{colorlinks=true,linkcolor=red, anchorcolor=green, citecolor=cyan, urlcolor=red, filecolor=magenta, pdftoolbar=true}

\newtheorem{theorem}{Theorem}[section]
\newtheorem{lemma}[theorem]{Lemma}
\newtheorem{proposition}[theorem]{Proposition}

\theoremstyle{definition}

\theoremstyle{remark}

\numberwithin{equation}{section}

\newcommand{\NN}{\mathbb{N}}

\newcommand{\CC}{\mathbb {C}}

\begin{document}
\setcounter{page}{1}
\title[Essential norm of the differential operator ]{Essential norm of the differential operator  }
\author [Tesfa  Mengestie]{Tesfa  Mengestie }
\address{ Western Norway University of Applied Sciences \\
Mathematicas section Sciences
Klingenbergvegen 8, N-5414 Stord, Norway}
\email{Tesfa.Mengestie@hvl.no}
\thanks{The author  is partially supported by HSH grant 1244/ H15.}
\begin{abstract}
This paper is a follow-up contribution to our work \cite{TM6} where we  studied some spectral properties of
the differential operator $D$ acting between generalized Fock spaces $\mathcal{F}_{(m,p)}$ and $\mathcal{F}_{(m,q)}$ when both exponents $p$ and $q$ are finite. In this note we continue to study the properties for the case when at least one of the  spaces is growth type.  We also  estimate
the   essential  norm of $D: \mathcal{F}_{(m,p)}\to \mathcal{F}_{(m,q)}$ for all $1\leq p, q\leq \infty$, and showed that if the operator fails to be compact, then its essential norm  is comparable to the operator norm and $\|D\|_e \simeq \big|m^{2+p}-m^{1+p}\big|^{\frac{1}{p}}\simeq   \|D\|.$
\end{abstract}

\maketitle

\section{Introduction} \label{1}
The differential operator $Df= f'$ is one of the fundamental operators in function related operator theory. However, the operator  is known to act in a discontinuous  fashion on many Banach spaces  including  on the classical Fock spaces, weighted Fock spaces where the weight decays faster than the classical Gaussian weight \cite{TM3}, on Fock--Sobolev spaces which are typical examples of weighted Fock spaces where the weight decays slower than the  Gaussian weight \cite{TM5}. In light of this, we  studied  the question of  how  slower   must  the weight function decay on generalized Fock spaces under which the operator $D$ admits some basic spectral structures \cite{TM6}, and found out that  the weight should in fact  decay much slower than the  classical Gaussian weight $e^{-|z|^2}$.   More precisely,  for $m>0$, we  considered  a class of generalized  Fock   spaces $\mathcal{F}_{(m,p)}$   which consist of all entire functions $f$ for which
\begin{align*}
\|f\|_{(m, p)}^p= \int_{\CC} |f(z)|^p e^{-p|z|^m} dA(z) <\infty,
\end{align*} where   $dA$ denotes the
usual Lebesgue area  measure on $\CC$.   Then for $0<p\leq q<\infty$,   it was proved that  $D: \mathcal{F}_{(m,p)} \to \mathcal{F}_{(m,q)}$ is bounded if and only if 
\begin{align}
\label{bounded}
m\leq 2-\frac{pq}{pq+q-p}
\end{align} and in this case  the norm is estimated  as
\begin{align}
\label{boundednorm}
\|D\|\simeq \begin{cases}
|m^{2+p}-m^{1+p}|^{\frac{1}{p}}\sup_{w\in \CC} \big(1+|w|\big)^{(m-1)+\frac{(q-p)(m-2)}{qp}}, & m\neq1.\\
 1,  & m=1
 \end{cases}
\end{align}
Compactness has been  described by the strict inequality \eqref{bounded} while the corresponding  equivalent condition for the case when  $p>q$  has been found to be \begin{align}
\label{bounded1}
m<1-2 \Big(\frac{1}{q}-\frac{1}{p}\Big)\end{align}  which is yet stronger than \eqref{bounded} and  in addition, forces boundedness to imply compactness.

One of the main purposes of this note is to  study  the situation when one of the Fock type spaces $\mathcal{F}_{(m,p)}$ is replaced by the natural growth type  spaces  $\mathcal{F}_{(m, \infty)}$ which consist of entire functions $f$ for which
\begin{align*}
 \|f\|_{(m, \infty)}= \sup_{z\in \CC} |f(z)|e^{-|z|^m}<\infty.
\end{align*}
For this,  our first result  below shows that  the weight function $|z|^m$ can grow  at most  as a complex polynomial of degree not exceeding $2-\frac{p}{p+1}$.
\begin{theorem}\label{thm1}
\begin{enumerate}
\item Let $0<p<\infty$ and $m>0$. Then  $D: \mathcal{F}_{(m,p)} \to \mathcal{F}_{(m,\infty)}$ is
\begin{enumerate}
\item bounded if and only if \begin{align}
\label{boundedinfinity}
m\leq  2-\frac{p}{p+1}\end{align} and  the norm is estimated by
\begin{align}
\label{boundedinfinitynorm}
\|D\| \overset{\footnotemark}{\simeq} \begin{cases}
|m^{2+p}-m^{1+p}|^{\frac{1}{p}}\sup_{w\in \CC} (1+|w|)^{\frac{m(p+1) -(p+2)}{p}}, & m\neq1\\
1, & m=1
\end{cases}.
\end{align}
\footnotetext{The  notation $U(z)\lesssim V(z)$ (or
equivalently $V(z)\gtrsim U(z)$) means that there is a constant
$C$ such that $U(z)\leq CV(z)$ holds for all $z$ in the set of a
question. We write $U(z)\simeq V(z)$ if both $U(z)\lesssim V(z)$
and $V(z)\lesssim U(z)$.}
\item compact if and only if\begin{align*}m< 2-\frac{p}{p+1}.\end{align*}
\end{enumerate}
\item Let $0<p< \infty$ and $m>0$. Then the following statements are equivalent.
\begin{enumerate}
\item $D: \mathcal{F}_{(m,\infty)} \to \mathcal{F}_{(m,p)}$ is bounded;
\item $D: \mathcal{F}_{(m,\infty)} \to \mathcal{F}_{(m,p)}$ is compact;
\item It holds that\begin{align}
\label{compact1}
 m <1-\frac{2}{p}.\end{align}
\end{enumerate}
\end{enumerate}
\end{theorem}
Before going further, we want to  remark that the  study in \cite{TM6} was initiated  in a quest for answering the question of  how fast should the associated weight function on generalized Fock  spaces decay  in order that  the operator  $D$ admits some basic spectral structures. Now   Theorem~\ref{thm1} and the corresponding result in \cite{TM6} provide  a clear description  for  its  decay, namely that the weight should decay in all cases  much slower than the classical Gaussian weight as  precisely specified in \eqref{bounded}, \eqref{bounded1}, \eqref{boundedinfinity}, and \eqref{compact1}. On the other hand, when $p=q= \infty$, as can be  seen from  \eqref{boundedinfinity},  $D: \mathcal{F}_{(m,\infty)} \to \mathcal{F}_{(m,\infty)}$ is bounded if and only if $m\leq 1$ and compactness is described by the strict inequality  $m<1$. This particular case follows  also  from many other  related works for example in \cite{Maria, Harutyunyan}.

 Another  purpose of this note is to show that  the differential operator on
 the generalized Fock spaces  $\mathcal{F}_{(m,p)}$ allows a small variety of qualitative non-compact behaviour compared to that of arbitrary boundedness. This rigidity property is measured using  essential
  norm of the operator and  showed that its essential norm is equivalent to the  operator norm. Recall that for two Banach spaces $\mathscr{H}_1$ and $\mathscr{H}_2$   the
 essential norm
$\|T\|_e$ of a bounded linear operator $T:\mathscr{H}_1  \to  \mathscr{H}_2 $
is defined as the distance from $T$ to the space of compact
operators  from  $\mathscr{H}_1$  to  $\mathscr{H}_2:$
\begin{align}
\label{def}
\|T\|_e= \inf_{K} \big\{\|T-K\|; \ \ K: \mathscr{H}_1 \to \mathscr{H}_2\text{ is a compact operator }\big\}.
\end{align}
In particular, \eqref{def} implies that  $T$ is compact if and only if its essential norm vanishes. Thus,  the essential norm can be interpreted as a quantity that  provides  a useful measure  for the  noncompactness of  operators. Several authors have studied such norms for several  operators on
 various functional spaces  including the Hardy spaces, Bergaman spaces, and Fock spaces; see for example \cite{ZZHH,ZZH,TM4,JR,Shapiro,SS, UEKI2}.  We prove the following estimates for  $D$ acting between the spaces $\mathcal{F}_{(m,p)}$.
\begin{theorem}\label{thm2}
 Let $1\leq p\leq q\leq \infty$ and  $D: \mathcal{F}_{(m,p)} \to \mathcal{F}_{(m,q)}$ is
 bounded. Then
\begin{align}
\label{estimate}
 \|D\|_e \simeq
 \begin{cases}|m^{2+p}-m^{1+p}|^{\frac{1}{p}}, \  \ \ \ \ q <\infty \ \text{and}\ m= 2-\frac{pq}{pq+q-p}\\
   1, \ \ \ \  \quad q = \infty= p, \  \text{and}\  m=1\\
    |m^{2+p}-m^{1+p}|^{\frac{1}{p}}, \   \ \ \ \ q = \infty, \ p<\infty,  \ \text{and}\  m= 2-\frac{p}{p+1}\\
 0,  \ \ \ \  \quad  \text{otherwise}.
  \end{cases}
\end{align}
\end{theorem}
We observe that if $D$ fails to be compact, then from Theorem~\ref{thm2} and the  relations  in \eqref{boundednorm} and \eqref{boundedinfinitynorm}, its essential norm can be simply estimated as $  \|D\|_e \simeq  |m^{2+p}-m^{1+p}|^{\frac{1}{p}}\simeq  \|D\|.$
\section{Preliminaries}
In this section we collect some basic facts and preliminary results that will be used  in the sequel. One of the  important ingredients needed is   the Littlewood--Paley type estimate from  \cite{Olivia2},
\begin{align}
\label{paley}
\|f\|_{(m, p)}^p \simeq |f(0)|^p + \int_{\CC} \frac{|f'(z)|^p e^{-p|z|^m}}{ (1+ |z|)^{p(m-1)}} dA(z),
\end{align}  which holds  for  all functions $f\in \mathcal{F}_{(m,p)}$.  On the other hand, for $p= \infty$, from  a simple modification of the  arguments  used in the proof of  Lemma~2.1 of \cite{TM3}, it follows that  $f$ belongs to the spaces $\mathcal{F}_{(m, \infty)}$ if and only if
\begin{align*}
\sup_{z\in \CC} \frac{|f'(z)|e^{-|z|^m}}{(1+|z|)^{m-1}} <\infty,
\end{align*} and in this case we estimate the norm  by
\begin{align}
\label{norm}
\|f\|_{(m, \infty)} \simeq |f(0)| + \sup_{z\in \CC} \frac{|f'(z)|e^{-|z|^m}}{(1+|z|)^{m-1}}.
\end{align}
Several properties of  linear operators  can often be described by their action on some special elements in the spaces.  The reproducing kernels   do often used for such purpose  in many  functional spaces.  Since  an explicit expression for the kernels in our current setting is still unknown, we will  use  another   sequence of special test functions. Such a sequence was first constructed in \cite{Borch} and  has been since then  used    by several authors  for example \cite{Olivia,TM3,JPP}.  We introduce the sequence as  follows.  We may first set
\begin{align*}\tau_m(z)=
\begin{cases}   1,  & 0\leq|(m^2-m)z|<1\\
\frac{|z|^{\frac{2-m}{2}}}{ |m^2-m|^{\frac{1}{2}}},\ \  & |(m^2-m)z|\geq 1.
 \end{cases}
 \end{align*}
 Then, for a sufficiently large positive number $R$, there exists a number $\eta(R)$ such that for any  $w\in \CC$ with $|w|> \eta(R)$, there exists an entire function $f_{(w, R)}$ such that
    \begin{align}
  \vspace{-0.3in}
  \label{test000}
      |f_{(w,R)}(z)| e^{-|z|^m}\leq C \min\Bigg\{ 1,\bigg(\frac{\min\{\tau_m(w), \tau_m(z)\}}{|z-w|}\bigg)^{\frac{R^2}{2}}\Bigg\}  \ \ \ \ \ \ \end{align}for all  $ z\in \CC$ and  for some constant $C$ that depends on $|z|^m$ and $R$. In particular when $z \in D(w, R\tau_m(w))$, the estimate becomes
   \begin{align}
  \label{test0}
   |f_{(w,R)}(z)| e^{-|z|^m}\simeq 1,
    \end{align} where $D(w,r)$ denotes  the Euclidean disk centered at $w$ and radius $r>0$.
 In addition, $f_{(w, R)}$ belongs to $\mathcal{F}_{(m,p)}$  with norm  estimated by
\begin{align}
\label{test}
\| f_{(w,R)}\|_{(m, p)}^p \simeq \tau_m^2(w),\ \ \ \  \eta(R) \leq |w|
\end{align} for all  $p$ in the range $0<p<\infty$. On the other hand, when $p= \infty$, from \eqref{test000} and \eqref{test0}, we easily deduce that
\begin{align}
\label{test00}
\| f_{(w,R)}\|_{(m, \infty)}\simeq 1.
\end{align}
Another important fact is the pointwise estimate for subharmonic functions $f$, namely that
\begin{align}
\label{pointwise}
|f(z)|^p e^{-p |z|^m} \lesssim \frac{1}{\sigma^2\tau_m^2(z)} \int_{D(z, \sigma \tau_m(z))} |f(w)|^pe^{-p|w|^m} dA(w)
\end{align} for all finite exponent $p$  and  a small positive number $\sigma$.  The estimate follows from
Lemma~2  of \cite{JPP}.

  Next, we   recall   the following useful covering lemma which is essentially from \cite{Olivia,OVL}.
\begin{lemma}\label{lem4}
 Let $\tau_m$ be as above.  Then, there exists a positive $\sigma >0$ and a sequence of points $z_j$ in $\CC$ satisfying the following conditions.
 \vspace{-0.2in}
 \begin{enumerate}
 \begin{multicols}{2}
 \item $z_j\not\in D(z_k,\sigma \tau_m(z_k)), \ \ j \neq k$;
 \item $\CC= \bigcup_jD(z_j, \sigma \tau_m(z_j))$;
  \end{multicols}
 \vspace{-0.2in}
 \item $\bigcup_{z\in  D(z_j, \sigma \tau_m(z_j))}D(z, \sigma \tau_m(z)) \subset D(z_j, 3\sigma \tau_m(z_j))$;
 \item The sequence $ D(z_j, 3\sigma \tau_m(z_j))$ is a covering of $\CC$ with finite multiplicity $N_{\max}$.
 \end{enumerate}
\end{lemma}
 The composition operator $C_\Phi f= f(\Phi)$ is one of the classical and  well studied objects in  function related operator theories. Our next result on $C_\Phi$ describes its ompactness property while acting on the spaces $\mathcal{F}_{(m,p)}$. The result  will not only play a vital role to prove our second main result in the previous section but also is interest of its own.
 \begin{proposition}\label{compact} Let $0<p\leq q \leq  \infty$ and $\Phi $ be a nonconstant entire function on $\CC$. Then  the composition operator
$C_\Phi:  \mathcal{F}_{(m,p)} \to  \mathcal{F}_{(m,q)}$  is compact if and only if   $\Phi(z)= az+b$ for some complex numbers $a$ and $b$ such that $|a|<1.$
    \end{proposition}
    \begin{proof}
     We begin with the proof of the necessity and  assume that $C_\Phi$ is compact. We also  bserve that the  normalized  sequence
          \begin{align}
\label{unitnorm}
f^*_{(w,R)}= \frac{f_{(w,R)}} {\|f_{(w,R)}\|_{(m,p)}}\simeq
\begin{cases}
 \frac{f_{(w,R)}}{\tau_m(w)^{\frac{2}{p}}},  & 1\leq p <\infty \\
  f_{(w,R)}, \ & p= \infty,
 \end{cases}
 \end{align} as described  from \eqref{test000}-\eqref{test00},
     converges to zero as $|w| \to \infty,$ and the convergence is uniform on compact subset of $\CC.$  Now if  $0<p<q= \infty,$  then  our assumption and   $C_\Phi$ applied to the sequence  $f^*_{(w, R)}$ imply
\begin{align*}
0= \lim_{|w| \to  \infty }\|C_\Phi f^*_{(w,R)}\|_{(m,\infty)} \simeq  \lim_{|w| \to  \infty }\sup_{z\in \CC}\frac{ |f_{(w, R)}(\Phi(z))|}{\tau_m(w)^{\frac{2}{p}}e^{|z|^m}}
 \geq  \lim_{|w| \to  \infty }\frac{ \big|f_{(w, R)}(\Phi(z))\big|}{\tau_m(w)^{\frac{2}{p}}e^{|z|^m}}
\end{align*}
for all $z, w\in \CC$. In particular,  setting $w= \Phi(z)$ and applying \eqref{test0} give
\begin{align}
\label{one}
0= \lim_{|\Phi(z)| \to  \infty }\big|f_{(\Phi(z),R)}(\Phi(z))\big|e^{-|\Phi(z)|^m} \frac{e^{|\Phi(z)|^m-|z|^m}}{\tau^{\frac{2}{p}}_m(\Phi(z))}
\simeq \lim_{|\Phi(z)| \to  \infty } \frac{e^{|\Phi(z)|^m-|z|^m}}{\tau^{\frac{2}{p}}_m(\Phi(z))}\nonumber\\
\simeq \lim_{|\Phi(z)| \to  \infty } e^{|\Phi(z)|^m-|z|^m-\frac{2}{p}\log\big(\tau_m(|\Phi(z)) \big)}
\end{align} from which we may  first  claim that $\Phi(z)= az+b$ for some complex numbers $a$ and $b$. If not,  there exists a sequence $z_k$ such that $|z_k| \to \infty $ and $\big|\frac{\Phi(z_k)}{z_k}\big|^m \to \infty$ as $k\to \infty$. It follows from this that  there exists an $N_1$ such that\begin{align*}
\Big|\frac{\Phi(z_k)}{z_k}\Big|^m-1-\frac{2}{p|z_k|^m}\log\big(\tau_m(|\Phi(z_k))\big)  \geq 1\end{align*} for all $k\geq N_1$. To this end, we have
\begin{align*}
e^{|\Phi(z_k)|^m-|z|^m-\frac{2}{p}\log\big(\tau_m(\Phi(z_k))\big )}= \sum_{n=0}^\infty \frac{1}{n\,!} \Big( |\Phi(z_k)|^m-|z_k|^m -\frac{2}{p}\log\big(\tau_m(\Phi(z_k)) \big)\Big)^n\nonumber\\
= \sum_{n=0}^\infty \frac{|z_k|^{nm}}{n\,!} \bigg( \frac{|\Phi(z_k)|^m}{|z_k|^m}-1-\frac{2}{p|z_k|^m}\log\big(\tau_m(\Phi(z_k)) \big)\bigg)^n \quad \quad \quad \quad \quad\nonumber\\
\geq \frac{|z_k|^{nm}}{n\,!} \bigg ( \frac{|\Phi(z_k)|^m}{|z_k|^m}-1-\frac{2}{p|z_k|^m}\log\big(\tau_m(\Phi(z_k))\big)\bigg)^n , \ \ k\geq N_1 \quad \quad \quad \quad \quad
\end{align*} from which we deduce
\begin{align}
\label{onee}
e^{|\Phi(z-k)|^m-|z|^m-\frac{2}{p}\log\big(\tau_m(\Phi(z))\big )} \gtrsim \frac{|z_k|^{nm}}{n\,!} \to \infty  \ \text{as}\   |z_k| \to \infty
\end{align} that  contradicts \eqref{one} and hence $\Phi(z)= az+b$.  We further claim that $|a|<1.$ If not, observe that the estimate in \eqref{one}  does in addition  imply
\begin{align}
\label{twoo}
\lim_{|\Phi(z)| \to  \infty } |\Phi(z)|^m-|z|^m-\frac{2}{p}\log\big(\tau_m(\Phi(z))\big)\quad \quad \quad \quad \quad  \quad \quad \quad \quad \quad\nonumber\\
=  \lim_{|az+b| \to  \infty } |az+b|^m-|z|^m-\frac{2}{p}\log\big(\tau_m(|az+b|)\big)<0
\end{align}  which holds only if $|a|<1$.

If  $p= q= \infty, $ then following the same arguments as those leading to  \eqref{one}, we  obtain
\begin{align}
\label{oneand}
0= \lim_{|\Phi(z)| \to  \infty }\big|f_{(\Phi(z),R)}(\Phi(z))\big|e^{-|\Phi(z)|^m}  e^{|\Phi(z)|^m-|z|^m}
\simeq \lim_{|\Phi(z)| \to  \infty } e^{|\Phi(z)|^m-|z|^m}
\end{align}  from which and  repeating the arguments leading to  \eqref{onee} and \eqref{twoo}, we easily arrive at the desired conclusion.

In a similar manner, when  $0<p\leq q<\infty$, then  applying \eqref{pointwise} we estimate
\begin{align*}
0=  \lim_{|w| \to \infty} \| C_\Phi f^*_{(w, R)}\|_{(m, q)} \geq \lim_{|w| \to \infty} \frac{1}{\tau_m(w)^{\frac{2q}{p}}} \int_{D(w,\sigma\tau_m(w))}\frac{ |f_{(w, R)}(\Phi(z))|^q}{ e^{q|z|^m}} dA(z) \quad \quad \quad \nonumber\\
=  \lim_{|w| \to \infty}  \frac{1}{\tau_m(w)^{\frac{2q}{p}}}  \int_{D(w,\sigma\tau_m(w))}\frac{ |f_{(w, R)}(z)|^q e^{-q|z|^m}}{ e^{q|\Phi^{-1}(z)|^m-q|z|^m}} dA(\Phi^{-1}(z)) \quad  \nonumber\\
\gtrsim \lim_{|w| \to \infty} \frac{1}{\tau_m(w)^{\frac{2q}{p}}} \tau_m(w)^2  e^{-q|\Phi^{-1}(w)|^m+q|w|^m} \quad  \quad \quad \quad \quad \quad \nonumber\\
=\lim_{|\Phi(w)| \to \infty} \tau_m(\Phi(w))^{2-\frac{2q}{p}}  e^{q|\Phi(w)|^m- q|w|^m}\quad  \quad \quad \quad \quad \quad\nonumber\\
\simeq \lim_{|\Phi(w)| \to \infty} e^{q|\Phi(w)|^m- q|w|^m + \frac{2(p-q)}{p} \log \big(\tau_m(\Phi(w))\big)}\quad  \quad \quad \quad \quad \quad
\end{align*} from which and repeating   the arguments leading to  \eqref{onee} and \eqref{twoo} again, we arrive at our assertion.

To prove the sufficiency of the condition, we let  $f_n$   be a uniformly bounded sequence of functions in $\mathcal{F}_{(m,p)}$ that converge uniformly to zero on   compact subsets of  $\CC$. Then we consider three different cases.

  \emph{Case 1}:  if $q= \infty$ and $0<p<\infty$, then  for a  positive number $r$ and eventually applying \eqref{pointwise}, we estimate
 \begin{align*}
 \|C_\Phi f_n\|_{(m,\infty)} = \sup_{z\in\CC} |f_n(\Phi(z))|e^{-|z|^m}=  \sup_{z\in\CC} |f_n(az+b)| e^{-|z|^m}\quad \quad \quad \quad  \quad \quad \quad \quad \nonumber\\
\lesssim \sup_{|az+b|>r} |f_n(az+b)|e^{-|az+b|^m} \frac{e^{|az+b|^m}}{e^{|z|^m}}+\sup_{|az+b|\leq r} |f_n(az+b)| e^{-|z|^m} \nonumber\\
 \lesssim  \|f_n\|_{(m,p)} \sup_{|az+b|>r}\frac{e^{|az+b|^m-|z|^m}}{\tau_m(az+b)^{\frac{2}{p}}}+ \sup_{|az+b|\leq r} |f_n(az+b)| \quad \nonumber\\
 \simeq \|f_n\|_{(m,p)} \sup_{|az+b|>r}\frac{e^{|az+b|^m-|z|^m}}{\tau_m(az+b)^{\frac{2}{p}}}
 + \sup_{|az+b|\leq r} |f_n(az+b)|,
    \end{align*} where in the last inequality we  used the pointwise estimate \eqref{pointwise}.
   Since $\|f_n\|_{(m,\infty)}$ is  uniformly bounded and $|a|<1$,  the first summand  above  goes to zero as $r \to \infty$ and   the second goes to zero when $n\to \infty$. This implies
$ \|C_\Phi f_n\|_{\mathcal{F}_{(m,\infty)}} \to 0$ as $n\to \infty $ from which our assertion follows in this case.

\emph{Case 2}:  if $q= \infty=p$, then  for a  positive number $r$, we  also have
\begin{align*}
 \|C_\Phi f_n\|_{(m,\infty)} = \sup_{z\in\CC} |f_n(\Phi(z))|e^{-|z|^m}=  \sup_{z\in\CC} |f_n(az+b)| e^{-|z|^m}\quad \quad \quad \quad \quad \quad  \nonumber\\
 \simeq \sup_{|az+b|>r} |f_n(az+b)|e^{-|az+b|^m} \frac{e^{|az+b|^m}}{e^{|z|^m}}+\sup_{|az+b|\leq r} |f_n(az+b)| e^{-|z|^m}\nonumber\\
 \lesssim  \|f_n\|_{(m,\infty)} \sup_{|az+b|>r}\frac{e^{|az+b|^m}}{e^{|z|^m}}+ \sup_{|az+b|\leq r} |f_n(az+b)|
     \end{align*} from which the claim follows.

     \emph{Case 3}: if $0<p\leq q<\infty$, then  applying \eqref{pointwise}
     \begin{align*}
 \|C_\Phi f_n\|_{(m,q)}^q =   \int_{\CC} |f_n(az+b)|^q e^{-q|az+b|^m} \Big( e^{q|az+b|^m-q|z|^m}\Big) dA(z) \quad \quad \quad \quad \quad \quad \quad \quad \nonumber\\
  \lesssim  \int_{\CC} \bigg(\frac{1}{\tau_m(az+b)^2} \int_{D(az+b, \sigma \tau_m(az+b))} \frac{ |f_n(w)|^p }{ e^{p|w|^m}}dA(w)\bigg)^{\frac{q}{p}} \frac{e^{q|az+b|^m}}{e^{q|z|^m}} dA(z).
  \end{align*}
  Now if $|az+b|>r$ for some positive number $r$, then the part of the  integral on $ \{z\in \CC: |az+b|>r\} $ is bounded by
  \begin{align*}
  \| f_n\|_{(m,p)}^q \int_{|az+b|>r} \frac{e^{q|az+b|^m-q|z|^m}}{\tau_m(az+b)^{\frac{2q}{p}}}   dA(z)
  \lesssim \int_{|az+b|>r} \frac{e^{q|az+b|^m-q|z|^m}}{\tau_m(az+b)^{\frac{2q}{p}}}  dA(z)
  \end{align*} which is finite as $|a|<1$ and tends to zero as $r\to \infty$.
  On the other hand, if  $|az+b|\leq r$, then  using the fact that $|a|<1$ we find that  the remaining   part of the integral  is bounded by
  \begin{align*}
 \sup_{|az+b|\leq r} |f_n(az+b)|^q \int_{|az+b|\leq r} e^{q|az+b|^m-q|z|^m} dA(z) \lesssim  \sup_{|az+b|\leq r} |f_n(az+b)|^q \to 0\
     \end{align*} as $ n\to \infty$  and completes the proof of the proposition.
    \end{proof}
\section{Proof of the Main results}
\subsection{Proof of Theorem~\ref{thm1}}
In this section we prove our first main results.\\
\textbf{Part i)}: Let  $0<p<\infty$ and $D:\mathcal{F}_{(m,p)} \to \mathcal{F}_{(m,\infty)}$ is bounded. Then, applying the sequence of function in \eqref{unitnorm}   we estimate
\begin{align*}
 \|D\| \gtrsim  \|Df^*_{(w, R)}\|_{(m,\infty)}\simeq \frac{ \sup_{z\in \CC}|f'_{(w, R)}(z)| e^{-|z|^m} }{\tau_m^{\frac{2}{p}}(w)}
\geq   \frac{|f'_{(w, R)}(w)| e^{-|w|^m}}{\tau_m^{\frac{2}{p}}(w)} \simeq \frac{m |w|^{m-1} }{\tau_m^{\frac{2}{p}}(w)}
\end{align*} for all $w\in \CC$. This  happens to hold  only if
\begin{align}
\label{infitynorm}
\|D\| \gtrsim \sup_{w\in \CC} \frac{m |w|^{m-1} }{\tau_m^{\frac{2}{p}}(w)}\simeq \begin{cases}
 \frac{\sup_{w\in \CC} (1+|w|)^{\frac{m(p+1) -(p+2)}{p}}}{m^{-1}|m^2-m|^{-\frac{1}{p}}}, & m\neq1\\
1, & m=1
\end{cases}
\end{align} from which our assertion and one side of the norm  estimate for $D$ follow.

Conversely, applying \eqref{paley} and \eqref{pointwise},  we  also have
 \begin{align}
  \|Df\|_{(m,\infty)} = \sup_{z\in \CC}\frac{ |f'(z)|}{e^{|z|^m}} \lesssim \sup_{z\in \CC}\Bigg( \frac{1}{\tau_m^2(z)} \int_{D(z, \sigma\tau_m(z))} \frac{|f'(w)|^p}{e^{p|w|^m}} dA(w)\Bigg)^{\frac{1}{p}} dA(z) \nonumber
 \end{align}
 Now for each point $z\in D(w, \sigma\tau_m(w))$, observe  that  $1+|z| \simeq 1+|w|$. Taking this into account,  we further estimate the above by
 \begin{align*}
 \sup_{z\in \CC}\bigg( \frac{m^p(1+|z|)^{p(m-1)}}{\tau_m^2(z)} \int_{D(z, \sigma\tau_m(z))} \frac{|f'(w)|^pe^{-p|w|^m}}{m^p(1+|w|)^{p(m-1)}} dA(w)\bigg)^{\frac{1}{p}} dA(z)\nonumber\\
\lesssim \begin{cases}
 \|f\|_{(m,p)}\frac{\sup_{w\in \CC} (1+|w|)^{\frac{m(p+1) -(p+2)}{p}}}{|m^{2+p}-m^{1+p}|^{-\frac{1}{p}}}, & m\neq1\\
\|f\|_{(m,p)}, & m=1
\end{cases}
  \end{align*} from which the sufficiency of the condition and
the reverse side of the estimate in \eqref{infitynorm} follow.

We now turn to the proof of the compactness part and  first assume that   $m<2-\frac{p}{p+1}$.  Then for  each positive $\epsilon$,  there exists $N_1$ such that
 \begin{align}
 \sup_{|w|> N_1}  |m^{2+p}-m^{1+p}|^{\frac{1}{p}}
 (1+|w|)^{\frac{m(p+1) -(p+2)}{p}} \simeq  \sup_{|w|> N_1}   \frac{m^p(1+ |w|)^{p(m-1)}}{\tau_m^{2}(w)}<\epsilon.
 \label{partly}
 \end{align}
 Next, we let   $f_n$ to  be a uniformly bounded sequence of functions in $\mathcal{F}_{(m,p)}$ that converges uniformly to zero on   compact subsets of  $\CC$.  Then applying   \eqref{paley} and  arguing in the same way as in the  series of estimations made  above, and invoking eventually \eqref{partly} it follows that
 \begin{align*}
  \frac{ |f_n'(z)|^p}{ e^{p|z|^m}}\lesssim \frac{1}{\tau_m^2(z)} \int_{D(z, \sigma\tau_m(z))} \frac{|f_n'(w)|^p}{e^{p|w|^m}} dA(w)
     =\frac{1}{\tau_m^{2}(z)}\int_{\substack{w\in  D(z, \sigma\tau_m(z))\\ |w|\leq N_1}}\frac{|f_n'(w)|^p}{ e^{p|w|^m}} dA(w) \\
     +  \int_{\substack{w\in D(z, \sigma\tau_m(z))\\|w|> N_1}} \frac{|f_n'(w)|^p e^{-p|w|^m}}{\tau_m^{2}(w)} dA(w)\quad \quad \quad \quad \quad \quad \quad \quad \quad \quad \quad \quad \quad  \nonumber\\
   \lesssim  \sup_{|w|\leq N_1} |f_n(w)|^p   + \|f_n\|_{(m,p)}^p  \sup_{|w|>N_1} \frac{m^p(1+ |w|)^{p(m-1)}}{\tau_m^{2}(w)} \quad \quad \quad \nonumber\\
   \lesssim  \sup_{|w|\leq N_1} |f_n(w)|^p+ \sup_{|w|>N_1} (1+|w|)^{m(p+1) -(p+2)}\quad \quad \nonumber\\
   \lesssim  \sup_{|w|>N_1} \frac{m^p(1+ |w|)^{p(m-1)}}{\tau_m^{2}(w)}  \lesssim  \epsilon^p \  \ \text{as}\ \ n\to \infty \quad  \quad
 \end{align*} and from which we have that
 \ \begin{align*}
    \|Df_n\|_{(m,\infty)}= \sup_{z\in\CC} |f_n'(z)| e^{-|z|^m} \lesssim \epsilon \ \ \text{as}\ n\to \infty.
 \end{align*}
     On the other hand, if   $D$ is compact, applying  the sequence of functions     $ f^*_{(w,R)}$ in \eqref{unitnorm}, \eqref{pointwise} and \eqref{test0}, we find
 \begin{align*}
 \frac{m^p(1+|w|)^{p(m-1)}}{\tau_m^2(w)} \simeq  m^p(1+|w|)^{p(m-1)}  e^{-q|w|^m} |f^*_{(w, \eta(R))}(w)|^p \quad \quad \quad \quad \quad \quad \quad \quad \\
 \lesssim  \int_{D(w, \sigma\tau_m(w))} m^p(1+|z|)^{p(m-1)} |f^*_{(w, \eta(R))}(z)|^pe^{-p|z|^m}dA(z) \quad \quad \quad \quad \quad \quad \quad  \nonumber\\
 \leq  \int_{D(w, \sigma\tau_m(w))} \bigg(\sup_{w\in \CC} m(1+|z|)^{(m-1)} |f^*_{(w, \eta(R))}(z)|e^{-|z|^m}\bigg)^p dA(z)
  \simeq   \|Df^*_{(w, \eta(R))}\|_{(m,\infty)}^p
 \end{align*} from which we have that
 \begin{align*}
|m^{2+p}-m^{1+p}|^{\frac{1}{p}}(1+|w|)^{m-1+ \frac{m-2}{p}} \simeq  \frac{(1+|w|)^{(m-1)}}{ \tau_m ^{\frac{2}{p}}(w) } \lesssim \|Df^*_{(w, \eta(R))}\|_{(m,\infty)} \to 0,
 \end{align*}as $ |w| \to \infty$  which holds only when  $m-1+ \frac{m-2}{p} <0$ as asserted.

     \textbf{Part ii)}: Since (b)  $\Rightarrow$ (a), we will verify that  (a)$\Rightarrow$ (c) and (c)$\Rightarrow$ (b). For the first we argue as follows.   Let $0<p<\infty$ and $R$  be a sufficiently large number   and $(z_k)$ be the covering sequence as in Lemma~\ref{lem4}. Then by  Lemma~2.4 of \cite{TM3}, the function
\begin{align*}
F= \sum_{z_k:|z_k|\geq\eta(R)} a_k f_{(z_k,R)} \in\mathcal{F}_{(m,\infty)} \ \text{and} \ \|F\|_{(m,\infty)} \lesssim \|(a_k)\|_{\ell^\infty}
\end{align*}  for every $\ell^\infty$  sequence $(a_k)$ .
If $(r_k(t))_k$  is the Radmecher sequence of function on $[0,1]$ chosen as in \cite{DL}, then the sequence $(a_kr_k(t))\in \ell^\infty$ with $\|(a_kr_k(t))\|_{\ell^\infty}= \|(a_k)\|_{\ell^\infty}$ for all $t$. This implies that the function
\begin{align*}
F_t= \sum_{z_k:|z_k|\geq\eta(R)} a_k r_k(t) f_{(z_k,R)}\in \mathcal{F}_{(m,\infty)} \ \text{and} \ \|F_t\|_{(m,\infty)} \lesssim \|(a_k)\|_{\ell^\infty}.
\end{align*} Then, an  application of Khinchine's inequality \cite{DL} yields
\begin{align}
\label{Khinchine1}
\Bigg(\sum_{z_k:|z_k|\geq\eta(R)} |a_k|^{2}|f'_{(z_k,R)}(z)|^2\Bigg)^{\frac{p}{2}}\lesssim \int_{0}^1\bigg| \sum_{z_k:|z_k|\geq\eta(R)} a_k r_k(t) f'_{(z_k,R)}(z)\bigg|^q dt.
\end{align}
Making use of \eqref{Khinchine1}, and subsequently Fubini's theorem, we have
 \begin{align}
\int_{\CC}\Bigg(\sum_{z_k:|z_k|\geq\eta(R)} |a_k|^{2} |f'_{(z_k,R)}(z)|^2\Bigg)^{\frac{p}{2}} e^{-p|z|^m}dA(z)\quad \quad \quad \quad \quad  \quad \quad \quad \quad \quad  \quad\nonumber\\
\lesssim \int_{\CC} \int_{0}^1\bigg| \sum_{z_k:|z_k|\geq\eta(R)} a_k r_k(t) f'_{(z_k,R)}(z)\bigg|^p dt \frac{dA(z)}{e^{p|z|^m}}\nonumber\\
=  \int_{0}^1 \int_{\CC}\bigg| \sum_{z_k:|z_k|\geq\eta(R)} a_k r_k(t) f'_{(z_k,R)}(z)\bigg|^p \frac{dA(z) dt}{e^{p|z|^m}}\simeq \int_{0}^1\|DF_t\|_{(m,p)}^pdt\lesssim \|(a_k)\|_{\ell^\infty}^p.
\label{series3}
\end{align}
Then, using  \eqref{test0} we get
\begin{align*}
\sum_{z_k:|z_k|\geq\eta(R)} |a_k|^{p}\int_{D(z_k, 3\sigma\tau_m(z_k))}  (1+|z|)^{-p(m-1)}dA(z) \quad \quad \quad \quad \quad \quad \quad \quad \quad \quad \quad \quad \quad \quad \quad \quad \quad \quad \quad \nonumber\\
\simeq  \sum_{z_k:|z_k|\geq\eta(R)} |a_k|^{p}\int_{D(z_k, 3\sigma\tau_m(z_k))} (1+|z|)^{p(m-1)} \frac{|f'_{(z_k, R)}(z) |^p e^{-p|z|^m}}{(1+|z|)^{p(m-1)}} dA(z) \quad \quad \quad \quad \quad \quad \quad \quad \quad \quad \quad \quad\nonumber\\
 \simeq \int_{\CC} \sum_{z_k:|z_k|\geq\eta(R)} |a_k|^{p}\chi_{D(z_k, 3\sigma\tau_m(z_k))}(z)| f'_{(z_k, R)}(z)|^p e^{-p|z|^m}dA(z)\quad \quad \quad \quad \quad \quad \quad \quad \quad \quad \quad\nonumber\\
\lesssim \max\{1, N_{\max}^{1-p/2}\} \int_{\CC}\Bigg(\sum_{z_k:|z_k|\geq\eta(R)} |a_k|^{2} |f_{(z_k,R)}'(z)|^2\Bigg)^{\frac{p}{2}}e^{-p|z|^m}dA(z)\lesssim \|(a_k)\|_{\ell^\infty}^p. \quad \quad \quad \quad \quad \quad \quad \quad\quad \quad \quad \quad
\end{align*}
Setting, in particular,   $a_k=1$  for all $k$ in the above series of estimates results in
\begin{align*}
\sum_{z_k:|z_k|\geq\eta(R)}\int_{D(z_k, 3\sigma\tau_m(z_k))}  (1+|z|)^{-p(m-1)}dA(z)< \infty.
\end{align*}
Now we take  a positive number $r\geq \eta(R)$ such that whenever  $z_k$ of the covering sequence belongs to
$\{|z|<\eta(R)\}$, then $D(z_k, \sigma\tau_m(z_k)) $ belongs to $\{|z|<\eta(R)\}$. Thus,
\begin{align}
\label{againn}
\int_{|w|\geq r} m^p(1+|w|)^{p(m-1)} dA(w) \simeq \int_{|w|\geq r} \frac{1}{\tau_m^2(w)}\int_{D(w,3\sigma\tau_m(w))} \frac{ dA(z)dA(w)}{(1+|z|)^{-p(m-1)}} \quad \quad \quad \quad\quad \quad \nonumber\\
\leq \sum_{|z_k|\geq\eta(R)}\int_{D(z_k, \sigma\tau_m(z_k))} \frac{1}{\tau_m^2(w)}\int_{D(w, 3\sigma\tau_m(w))} \frac{dA(z) dA(w)}{m^{-p}(1+|z|)^{-p(m-1)}}\quad \quad \quad \quad \quad \quad \quad \quad  \nonumber\\
\lesssim \sum_{z_k:|z_k|\geq\eta(R)}\int_{D(z_k, 3\sigma\tau_m(z_k))} \frac{ dA(z)}{(1+|z|)^{-p(m-1)}} < \infty.\quad \quad \quad \quad \quad \quad \quad \quad \quad
\end{align}
It follows that
\begin{align*}
\int_{|w|< r} \frac{1}{\tau_m^2(w)}\int_{D(w, 3\sigma\tau_m(w))} m^p(1+|z|)^{p(m-1)} dA(z)dA(w)< \infty
\end{align*} from which and taking into account  \eqref{againn}, we obtain
\begin{align*}
\int_{\CC}  m^p(1+|z|)^{p(m-1)} dA(z)<\infty
\end{align*}
 which  holds only if  $p(m-1) < -2$  as asserted.

It remains to show that condition (c) implies (b).  To this end, let   $f_n$   be a uniformly bounded sequence of functions in $\mathcal{F}_{(m,\infty)}$ that converges uniformly to zero on   compact subsets of  $\CC$, and by the given condition, for each $\epsilon >0,$ there exists a positive number $r_1$ such that
 \begin{align*}
 \int_{|z|>r_1} (1+|z|)^{p(m-1)} dA(z) < \epsilon.
 \end{align*}
 It follows from this and  \eqref{norm} that
  \begin{align}
  \label{part}
\int_{|z|>r_1}|f_n'(z)|^p e^{-p|z|^m}dA(z)= \int_{|z|>r_1} \frac{|f_n'(z)|^pe^{-p|z|^m}}{(1+|z|)^{p(m-1)}} (1+|z|)^{p(m-1)}dA(z) \nonumber\\
  \lesssim \|f\|_{(m,\infty)}^p\int_{|z|>r_1} (1+|z|)^{p(m-1)}dA(z)
  \lesssim \|f\|_{(m,\infty)}^p  \epsilon \lesssim \epsilon.
  \end{align}
  On the other hand,  when $|z|\leq r_1$  we find
  \begin{align*}
  \int_{|z|\leq r_1} |f_n'(z)|^p e^{-p|z|^m}dA(z)\lesssim \int_{|z|\leq r_1} |f_n(z)|^p(1+|z|)^p  e^{-q|z|^m}dA(z) \quad \quad \nonumber\\
   \lesssim \sup_{|z|\leq r_1}|f_n(z)|^p  \int_{|z|\leq r_1} (1+|z|)^p  e^{-p|z|^m}dA(z)
   \lesssim \sup_{|z|\leq r_1}|f_n(z)|^p  \to 0 \ \
    \end{align*} as $ n\to \infty$ and,  from  which and \eqref{part} our claim $ \|Df_n\|_{(m,p)} \to 0$ as $n \to \infty$ follows.
\subsection{Proof of Theorem~\ref{thm2}}
In this section we prove our main results on essential norm.   Assume  that $1\leq p\leq \infty$ and $D:\mathcal{F}_{(m,p)} \to \mathcal{F}_{(m,q)}$ is bounded.  If $ q <\infty \ \text{and}\ m<2-\frac{pq}{pq+q-p} \ \text{or}\  q = \infty \ \text{and}\  m<2-\frac{p}{p+1}$, then as noticed before $D$ becomes a compact operator and its essential norm vanishes.  Thus, our aim here  is to  establish the result only for the two remaining cases namely   for
 $q <\infty \ \text{and}\ m= 2-\frac{pq}{pq+q-p}$ and  for $ q = \infty \ \text{and}\  m= 2-\frac{p}{p+1}$.
\subsection{Proof of the lower estimates in \eqref{estimate}}
To prove the lower bounds we will again use  the sequence of   functions in  \eqref{unitnorm}, and applying $D$ to such a sequence we find
\begin{align*}
\|D\|_e \geq \limsup_{|w|\to \infty} \|D f^*_{(w,R)}\|_{(m,q)}.
\end{align*}
Now if   $p=q= \infty $, then  making use of  \eqref{test0} we obtain
 \begin{align*}
\|D\|_e \geq \limsup_{|w|\to \infty} \| D f^*_{(w,R)}\|_{(m,\infty)} \simeq \limsup_{|w|\to \infty} \ \sup_{z\in \CC}
|f'_{(w,R)}(z)|e^{-|z|^m}\quad \quad \quad \quad \quad \quad\nonumber\\
\geq \limsup_{|w|\to \infty} |f'_{(w,R)}(w)|e^{-|w|^m}  \geq \limsup_{|w|\to \infty} m (1+|z|)^{m-1}
 \simeq  1,
\end{align*}  where we set m= 1 and from  which  the assertion follows. If we instead consider  $1\leq p<q =\infty $, then   it follows from  \eqref{test0}  and eventually setting $ m=2-\frac{p}{p+1}$ that
\begin{align*}
\|D\|_e\geq \limsup_{|w|\to \infty} \| D f^*_{(w,R)}\|_{(m,\infty)} \geq  \limsup_{|w|\to \infty} \frac{ |f'_{(w,R)}(w)|e^{-|w|^m}}{ \tau_m^{\frac{2}{p}}(w)}
\simeq  \limsup_{|w|\to \infty} \frac{m(1+|w|)^{m-1} }{\tau_m^{\frac{2}{p}}(w)}\nonumber\\
\simeq m|m^2-m|^{\frac{1}{p}} \limsup_{|w|\to \infty} (1+|w|)^{m-1+ \frac{m-2}{p}}=\big |m^{2+p}-m^{1+p}\big|^{\frac{1}{p}}.
 \end{align*}
On the other hand, if  $1\leq p\leq q <\infty$, then  making use of \eqref{test}  we again estimate
\begin{align*}
\|D\|_e \geq \limsup_{|w|\to \infty} \big\| D f^*_{(w,R)}\big\|_{(m,q)}
\simeq \limsup_{|w| \to \infty} \frac{1}{\tau_m^{\frac{2}{p}}(w)} \bigg(\int_{\CC}
\frac{|f'_{(w,R)}(z)|^q}{e^{q|z|^m}}dA(z)\bigg)^{\frac{1}{q}}\nonumber\\
 \geq \limsup_{|w|\to \infty} \frac{1}{\tau_m^{\frac{2}{p}}(w)} \bigg( \int_{D(w,\sigma \tau_m(w))} |f'_{(w,R)}(z)|^qe^{-q |z|^m}dA(z)\bigg)^{\frac{1}{q}}\nonumber
 \end{align*} for some small  positive number $\sigma$.  
  An application of \eqref{test0}  and also  setting $ m= 2-\frac{pq}{pq+q-p}$ imply
 that the last term above is comparable to
 \begin{align*}
 \limsup_{|w|\to \infty}  \frac{1}{\tau_m(w)^{\frac{2}{p}}}\bigg(\int_{D(w, \sigma \tau_m(w))} m^q(1+|z|)^{q(m-1)}dA(z)\bigg)^{\frac{1}{q}}\quad \quad \quad \quad \quad \quad \nonumber\\
 \simeq m |m^2-m|^{\frac{1}{p}}\limsup_{|w|\to \infty}  (1+|w|)^{m-1+ \frac{(m-2)(q-p)}{pq}} \simeq \big |m^{2+p}-m^{1+p}\big|^{\frac{1}{p}}
 \end{align*}    which  completes the proof of the lower estimate in \eqref{estimate}.
\subsection{Proof of the upper estimates in \eqref{estimate}} For this, we may apply Proposition~\ref{compact} and  consider  a sequence of compact composition operators  $C_{\Phi_k}$ where  $\Phi_k(z)=\frac{k}{k+1} z$ for each $k\in \NN$.  Since  $D$ is bounded, then $D\circ C_{\Phi_k}:\mathcal{F}_{(m,p)} \to \mathcal{F}_{(m,q)}$ also  constitutes a sequence of compact operators. Then we may consider two different cases.

\emph{Case 1:} If  $q= \infty,$ then  we have
\begin{align}
\label{sum0}
\|D\|_e \leq \|D-D\circ C_{\Phi_k}\|= \sup_{\|f\|_{(m,p)}\leq 1} \|(D-D\circ C_{\Phi_k})f\|_{(m,\infty)}\quad \quad \quad\nonumber\\
\simeq \sup_{\|f\|_{(m,p)}\leq 1}\   \sup_{|z|>r}\Big|f'(z)- \frac{k}{k+1}f'(\Phi_k(z))\Big|e^{-|z|^m}\nonumber\\
+\sup_{\|f\|_{(m,p)}\leq 1} \ \sup_{|z|\leq r}\Big|f'(z)-\frac{k}{k+1}f'(\Phi_k(z))\Big|e^{-|z|^m}
\end{align} for a certain fixed positive number $r$.   If  in addition $p= \infty$,  the first summand above is bounded by
\begin{align*}
\sup_{\|f\|_{(m,p)}\leq 1}\   \sup_{|z|>r}\Big( \frac{k}{k+1}\Big|f'(z)- f'(\Phi_k(z))\Big|e^{-|z|^m}+ \frac{1}{k+1} |f'(z)|e^{-|z|^m}\Big)\quad \quad \quad \quad \quad \nonumber\\
\leq \sup_{\|f\|_{(m,\infty)}\leq 1}   \sup_{|z|>r}m(1+|z|)^{m-1}\Big( \frac{\big|f'(z)- f'(\Phi_k(z))\big|e^{-\psi(z)}}{m(1+|z|)^{m-1}}+ \frac{1}{k+1} \frac{|f'(z)|e^{-|z|^m}}{m(1+|z|)^{m-1}}\Big) \nonumber\\
\leq  \sup_{|z|>r} m(1+|z|)^{m-1} +  \frac{1}{k+1} \sup_{|z|>r}  m(1+|z|)^{m-1}\lesssim \sup_{|z|>r} m(1+|z|)^{m-1} = 1,
\end{align*} where the last equality follows when we set $m= 1.$

 Similarly,  if  $1\leq p <\infty$, then  it follows from  \eqref{pointwise}  and eventually setting $m= 2-\frac{p}{p+1}$   that the first summand  in \eqref{sum0} is  bounded by
\begin{align}
\label{last}
\sup_{\|f\|_{(m,p)}\leq 1}  \sup_{|z|>r}  \frac{1}{ \tau_m^{\frac{2}{p}}(z)}\bigg( \int_{D(z, \sigma \tau_m(z))}\bigg(\frac{k^p \big|f'(w)- f'(\Phi_k(w))\big|^p + |f'(w)|^p}{e^{p|z|^m}(k+1)^p}\bigg)  dA(w)\bigg)^{\frac{1}{p}}\quad\quad\quad\quad \quad \quad\nonumber\\
\lesssim \sup_{\|f\|_{(m,p)}\leq 1}  \sup_{|z|>r}  \|f\|_{(m,p)}\bigg( \frac{ m(1+|z|)^{m-1}(k+1)}{ \tau_m^{\frac{2}{p}}(z)}\bigg) \quad\quad\quad \quad \quad\quad\quad\quad \quad \quad\quad\quad\quad\quad \quad \quad\nonumber\\
\leq|m^{2+p}-m^{1+p}|^{\frac{1}{p}}\sup_{|z|>r}  (1+|z|)^{m-1+ \frac{m-2}{p}}\Big(  \frac{k+2}{k+1}\Big)\quad\quad\quad\quad \quad \quad \quad \quad \quad\quad \quad\nonumber\\
 \lesssim |m^{2+p}-m^{1+p}|^{\frac{1}{p}}\sup_{|z|>r}  (1+|z|)^{m-1+ \frac{m-2}{p}}\simeq |m^{2+p}-m^{1+p}|^{\frac{1}{p}}.\quad \quad \quad\quad \quad \quad\quad
\end{align}
As for the second summand in \eqref{sum0}, we observe  that by integrating the function $f''$ along
 the radial segment $ [\frac{kz}{k+1}z, z]$   we find
\begin{align*}
\bigg|f'(z)-f'\Big(\frac{k}{k+1}z\Big)\bigg| \leq \frac{|z||f''(z^*)|}{k+1}
\end{align*}  for some $z^*$ in  the radial segment $ [\frac{kz}{k+1}z, z]$. By Cauchy estimate's for $f''$, we  also have
\begin{align*}
|f''(z^*)| \leq \frac{1}{r} \max_{|z|=2r} |f'(z)|,
\end{align*}
and hence
\begin{align}
\label{cauchy}
\bigg|f'(z)- \frac{k}{k+1}f'\Big(\frac{k}{k+1}z\Big)\bigg| e^{-|z|^m} \leq \frac{k}{k+1}\Big|f'(z)- f'\big(\frac{k}{k+1}z\big)\Big| e^{-|z|^m}\nonumber\\
+ \frac{|f'(z)| e^{-|z|^m}}{k+1}
\lesssim\frac{|z|e^{-|z|^m}}{r(k+1)}  \max_{|z|=2r} |f'(z)|  + \frac{|f'(z)| e^{-|z|^m}}{k+1}
\end{align} from which and if $p=\infty$ and applying \eqref{norm}, then the second summand in \eqref{sum0} is bounded by
\begin{align*}
\sup_{\|f\|_{(m,\infty)}\leq 1} \ \sup_{|z|\leq r}\bigg(\frac{|z|e^{-|z|^m}}{r(k+1)}  \max_{|z|=2r} |f'(z)|  + \frac{|f'(z)| e^{-|z|^m}}{k+1}\bigg)
\lesssim \frac{m(1+|r|)^{m-1}}{k+1}  = \frac{1}{k+1} \to 0
\end{align*} as $ \ k \to \infty$ and when  we set $m= 1$ here again.

On the other hand, if $1\leq p<\infty,$ we may further make some estimations in \eqref{cauchy}. By \eqref{pointwise} and \eqref{paley} and eventually setting $m= 2-\frac{p}{p+1}= \frac{p+2}{p+1}$, we  have
\begin{align*}
\max_{|z|=2r} |f'(z)|\lesssim \max_{|z|=2r}\frac{ e^{|z|^m} m (1+|z|)}{\tau_m^{\frac{2}{p}}(z)} \Bigg(\int_{D(z, \sigma\tau_m(z))}\frac{ |f'(w)|^p e^{-p|w|^m}}{m^p (1+|w|)^p} dA(w)\Bigg)^{\frac{1}{p}}\nonumber\\
\lesssim \|f\|_{(m,p)} \max_{|z|=2r}\frac{ e^{|z|^m} m |m^2-m|^{\frac{1}{p}}(1+|z|)^{m-1}}{\tau_m^{\frac{2}{p}}(z)} \lesssim \|f\|_{(m,p)} e^{(2r)^m} m(1+|r|)^{m-1+ \frac{m-2}{p}}\nonumber\\
= \|f\|_{(m,p)} e^{(2r)^m} |m^{2+p}-m^{1+p}|^{\frac{1}{p}}.\quad \quad \quad \quad \quad \quad \quad \quad
\end{align*}
Now combining all the above estimates, we see that the second piece of the sum in \eqref{sum0} is bounded by
\begin{align*}
\sup_{\|f\|_{(m,p)}\leq 1} \ \sup_{|z|\leq r}\big|f'(z)-\frac{k}{k+1}f'(\Phi_k(z))\big|e^{-|z|^m}
\lesssim \frac{p+2}{(k+1)(p+1)}\frac{\sup_{\|f\|_{(m,p)}\leq 1}\|f\|_{(m,p)}}{  e^{-(2r)^m}}\nonumber\\
 \leq \frac{1}{k+1} e^{(2r)^m} \rightarrow 0 \ \text{as} \ k \rightarrow \infty,
\end{align*}
 from which, \eqref{last}, $m=\frac{p+2}{p+1}$  and since $r$ is arbitrary, we  deduce
\begin{align*}
\|D\|_{e}\lesssim \sup_{|z|>r}  |m^{2+p}-m^{1+p}|^{\frac{1}{p}}(1+|z|)^{m-1+ \frac{m-2}{p}}\quad \quad \quad \quad \quad \quad \quad \quad \quad \quad \quad \quad \quad \quad \quad \quad\nonumber\\
= |m^{2+p}-m^{1+p}|^{\frac{1}{p}} \limsup_{|z| \to \infty} (1+|z|)^{m-1+ \frac{m-2}{p}}=|m^{2+p}-m^{1+p}|^{\frac{1}{p}} \quad \quad \quad\quad
\end{align*} and completes the first case.

\emph{Case 2: }When   $1\leq p\leq q <\infty$, then  we argue as follows. We may first estimate
\begin{align}
\label{upper1}
\|D\|_e\leq \|D-D\circ C_{\Phi_k}\|= \sup_{\|f\|_{(m,p)}\leq 1} \|(D-D\circ C_{\Phi_k})f\|_{(m,q)}\quad \quad \quad \nonumber\\
\simeq \sup_{\|f\|_{(m,p)}\leq 1} \Bigg(\int_{\CC}\big|f'(z)-f'(\Phi_k(z)) \Phi'_k(z)\big|^q e^{-q|z|^m}dA(z)\Bigg)^{\frac{1}{q}}.
\end{align}
Applying Lemma~\ref{lem4} and estimate  \eqref{pointwise}, we get
 \begin{align}
 \int_{\CC}\frac{\big|f'(z)-\frac{k}{k+1}f'(\Phi_k(z))\big|^q}{ e^{q|z|^m}}dA(z)\leq \sum_j \int_{D(z_j, \sigma \tau_m(z_j))}\frac{ \big|f'(z)-\frac{k}{k+1}f'(\Phi_k(z))\big|^q}{  e^{q|z|^m}}dA(z)\quad \quad \quad \quad \quad \quad\quad \quad \quad \quad \quad \nonumber\\
 \lesssim \sum_j  \int_{D(z_j, \sigma \tau_m(z_j))}\Bigg( \int_{D(z, \sigma \tau_m(z))}  \frac{\big|f'(w)-\frac{k}{k+1}f(\Phi_k(w))\big|^p}{ e^{p|w|^m}} dA(w)\Bigg)^{\frac{q}{p}} \frac{dA(z)}{\tau_m^{\frac{2q}{p}}(z)}  \quad \quad  \quad \quad \quad \quad \quad\quad \quad\quad \quad \nonumber\\
 \lesssim  \sum_j  \Bigg( \int_{D(z_j, 3\sigma\tau_m(z_j))} \frac{\big|f'(w)-\frac{k}{k+1}f'(\Phi_k(w))\big|^p}{ e^{p |w|^m}m^p(1+|w|)^{p(m-1)}} dA(w) \Bigg)^{\frac{q}{p}} \quad \quad \quad \quad \quad\quad\quad \quad \quad \quad \quad\quad\nonumber\\
\times \int_{D(z_j, \sigma\tau_m(z_j))} \frac{m^q(1+|z|)^{q(m-1)}}{\tau_m^{\frac{2q}{p}}(z)} dA(z). \quad \quad  \quad \quad\quad \quad \quad \quad  \quad \quad\quad \quad \nonumber
\end{align}
We spilt now  the above sum as
\begin{align}
\label{sum}
\sum_j = \sum_{j: |z_j|>r} +\sum_{j: |z_j|\leq r}
\end{align} for some fixed positive number $r$ again. Then  applying
  Minkowski inequality (since $q\geq p,$) and the finite multiplicity $N$ of the covering sequence $D(z_j, 3\sigma\tau(z_j))$ and setting $m=2-\frac{pq}{pq+q-p} $,  the first sum is bounded by
\begin{align*}
\sup_{j: |z_j| >r }\Bigg(\int_{D(z_j, \sigma \tau_m(z_j))} \frac{m^q(1+|z|)^{q(m-1)}}{\tau_m(z)^{\frac{2q}{p}}} dA(z)\Bigg)  \quad \quad \quad \quad \quad \quad \quad \quad \quad\quad \quad \quad  \quad \quad\quad \quad \quad \quad \quad \quad \quad \quad\quad \nonumber\\
 \times  \Bigg(  \int_{D(z_j, 3\sigma\tau_m(z_j))} \frac{\big|f'(w)-\frac{k}{k+1}f'(\Phi_k(w))\big|^p}{ e^{p |w|^m}m^p(1+|w|)^{p(m-1)}} dA(w) \Bigg)^{\frac{q}{p}}\quad \quad \quad \quad\quad \quad \quad \quad \quad \quad \quad  \nonumber\\
\lesssim \|f\|_{(m,p)}^q \Big(1+ \frac{1}{k+1}\Big) \sup_{|z_j| >r }  \int_{D(z_j, \sigma \tau_m(z_j))} \frac{m^q(1+|z|)^{q(m-1)}}{\tau_m(z)^{\frac{2q}{p}}} dA(z).\quad \quad \quad \quad\quad \quad \quad \quad \quad  \nonumber\\
\simeq  \|f\|_{(m,p)}^q \sup_{j: |z_j| >r } m^q|m^2-m|^{\frac{q}{p}} (1+|z_j|)^{q(m-1)+ (m-2) \frac{q-p}{p}}\quad \quad\quad \quad \quad \quad\quad \quad\quad \quad\quad  \nonumber\\
\lesssim m^q|m^2-m|^{\frac{q}{p}} \sup_{j: |z_j| >r }  (1+|z_j|)^{q(m-1)+ (m-2) \frac{q-p}{p}} =m^q|m^2-m|^{\frac{q}{p}}  \quad\quad \quad \quad \quad\quad \quad\quad
\end{align*} where  we, in particular, used  that $\|f\|_{\mathcal{F}_p^\psi}\leq 1$.

We plan to show that  the second sum in \eqref{sum} tends to zero when $k \to \infty$. Then since $r$ is arbitrary,  our upper estimate will follow from  the series of estimates we made  starting from \eqref{upper1}. To this end, as done before, making use of \eqref{cauchy} and Minkowski inequality again,  we estimate
\begin{align*}
\sum_{j: |z_j|\leq r}  \Bigg( \int_{D(z_j, 3\sigma\tau_m(z_j))} \frac{ \big|f'(w)-\frac{k}{k+1}f'(\Phi_k(w))\big|^p}{m^p(1+|w|)^{p(m-1)} e^{p|w|^m}} dA(w) \Bigg)^{\frac{q}{p}}\quad \quad \quad \quad\quad \quad\quad \quad \nonumber\\
\times \int_{D(z_j, \sigma \tau_m(z_j))} \frac{m^q(1+|z|)^{q(m-1)}}{\tau_m^{\frac{2q}{p}}(z)}dA(z)\quad  \nonumber\\
\lesssim \Bigg(\sum_{j: |z_j|\leq r}   \int_{D(z_j, 3\sigma\tau_m(z_j))} \frac{|w|^p \big(\max_{|w|=2r} |f'(w)|\big)^p+ |f'(w)|^p }{r(k+1)^p (m^p(1+|w|)^{p(m-1)})e^{p|w|^m} } dA(w) \Bigg)^{\frac{q}{p}} \nonumber\\
\times
\int_{D(z_j, \sigma \tau_m(z_j))}\frac{m^q(1+|z|)^{q(m-1)}}{\tau_m^{\frac{2q}{p}}(z)} dA(z).
\end{align*}
Now we also have
\begin{align*}|w|\leq |w-z_j| +|z_j| \leq r+ \sigma \tau_m(z_j) \leq r+ \delta \sup_{|z_j|\leq r}\tau_m(z_j)
\lesssim r+ \delta r^{\frac{2-m}{2}}\leq 2r\end{align*} from which  we have that
the  preceding sum  is bounded  by
\begin{align*}
 \frac{\|f\|_{(m,p)}^q }{(1+k)^q} \sup_{|z_j|\leq r} \int_{D(z_j, \sigma \tau_m(z_j))}\frac{m^q(1+|z|)^{q(m-1)}}{\tau_m(z)^{\frac{2q}{p}}} dA(z)\quad \quad \quad \quad \quad \quad \quad \quad \quad \quad \quad \quad \quad \quad \quad\quad \quad \quad  \quad \quad\quad \quad \quad \quad\quad \quad\nonumber\\
 \lesssim  \frac{m^q|m^2-m|^{\frac{q}{p}}  }{(1+k)^q}\sup_{j: |z_j| \leq r }  (1+|z_j|)^{q(m-1)+ (m-2) \frac{q-p}{p}}
\lesssim  \frac{m^q|m^2-m|^{\frac{q}{p}}}{(1+k)^q}\rightarrow 0 \ \ \text{as} \ \ k\rightarrow \infty, \quad \quad \quad \quad\quad \quad\quad \quad \quad \quad \quad \quad\quad \quad \quad
 \end{align*} where the last inequality follows after setting $m= 2-\frac{pq}{pq+q-p}$ again. From this and series of estimates made above  we deduce that
 \begin{align*}
  \|D\|_{e}\lesssim  |m^{2+p}-m^{1+p}|^{\frac{1}{p}}\sup_{j: |z_j| >r } (1+|z_j|)^{(m-1)+ (m-2) \frac{q-p}{qp}}= |m^{2+p}-m^{1+p}|^{\frac{1}{p}}.
  \end{align*}

\end{document}